\documentclass{amsart}
\usepackage{amssymb,amstext,amsmath,amscd,amsthm,amsfonts,enumerate,graphicx,latexsym,stmaryrd,multicol}
\usepackage[usenames]{color}
\usepackage{setspace}
\usepackage{bm}
\usepackage[all]{xy}

\newtheorem{thm}{Theorem}[section]
\newtheorem{lem}[thm]{Lemma}
\newtheorem{prop}[thm]{Proposition}
\newtheorem{cor}[thm]{Corollary}
\theoremstyle{definition}

\newtheorem{conj}[thm]{Conjecture}

\newtheorem{rmk}[thm]{Remark}

\theoremstyle{remark}

\newtheorem*{ac}{Acknowledgments}
\newtheorem*{conv}{Convention}
\newtheorem*{proof of claim}{Proof of Claim}

\numberwithin{equation}{thm}

\def\cok{\operatorname{Coker}}
\def\depth{\operatorname{depth}}
\def\Ext{\operatorname{Ext}}

\def\G{\mathcal{G}}
\def\ge{\geqslant}
\def\height{\operatorname{ht}}
\def\Hom{\operatorname{Hom}}
\def\image{\operatorname{Im}}

\def\ker{\operatorname{Ker}}
\def\le{\leqslant}

\def\m{\mathfrak{m}}

\def\P{\mathbb{P}}
\def\p{\mathfrak{p}}

\def\Q{\mathbb{Q}}

\def\s{\mathrm{S}}
\def\spec{\operatorname{Spec}}

\def\syz{\Omega}

\def\Tor{\operatorname{Tor}}
\def\tr{\operatorname{Tr}}
\def\V{\mathrm{V}}
\def\x{\operatorname{X}}
\def\xx{\boldsymbol{x}}

\def\RHom{\operatorname{\mathbf{R}Hom}}
\def\S{\mathsf{S}}
\begin{document}
\allowdisplaybreaks
\title{Auslander--Reiten conjecture for normal rings}
\author{Kaito Kimura}
\address{Graduate School of Mathematics, Nagoya University, Furocho, Chikusaku, Nagoya 464-8602, Japan}
\email{m21018b@math.nagoya-u.ac.jp}
\thanks{2020 {\em Mathematics Subject Classification.} 13D07}
\thanks{{\em Key words and phrases.} (Auslander) transpose, Auslander--Reiten conjecture, Ext module, syzygy, Serre's condition, Tor module}

\begin{abstract}
In this paper, sufficient conditions for finitely generated modules over a commutative noetherian ring to be projective are given in terms of vanishing of Ext modules.
One of the main results of this paper asserts that the Auslander--Reiten conjecture holds true for every normal ring.
\end{abstract}
\maketitle
\section{Introduction}

Throughout the present paper, we assume that all rings are commutative and noetherian, that $R$ is a ring, and that $M$ is a finitely generated $R$-module.
Characterizing the projectivity of a module in terms of vanishing of Ext modules has been actively studied in ring theory; see \cite{Ar, secm, DEL, GT, HL, KOT, ST} for instance.
The main result of this paper is the following theorem.

\begin{thm}[Corollary \ref{S2 free height less than 2}]\label{AR conjecture (S2) introduction1}
Suppose that $R$ satisfies Serre's condition $(\S_2)$ and $M$ is locally of finite projective dimension in codimension one.
Then $M$ is a projective $R$-module if one of the following three conditions is satisfied, where $d=\dim R$.
\begin{enumerate}[\rm(1)]
\item
$\Ext_R^i(M,R)=\Ext_R^j(M^{\ast},M^{\ast})=0$ for all $1\le i\le d$ and $1\le j\le d-1$.
\item
$\Ext_R^i(M^{\ast},R)=\Ext_R^j(M,M)=0$ for all $1\le i\le d$ and $1\le j\le d-1$, and $M$ satisfies $(\S_2)$.
\item
$\Ext_R^i(M,R)=\Ext_R^j(M,M)=0$ for all $1\le i\le 2d+1$ and $1\le j\le d-1$.
\end{enumerate}
\end{thm}

We can actually give a more general condition than (3); see Corollary \ref{S2 free height less than 2}(3).
Araya \cite{Ar} proved the same assertion as Theorem \ref{AR conjecture (S2) introduction1}(3) under the assumption that $R$ is Gorenstein.
Also, Kimura, Otake, and Takahashi \cite{KOT} extended Araya's theorem to the case where $R$ is Cohen--Macaulay.
The above theorem further refines these results.
Moreover, the above theorem and some of the results given as its corollary refine or recover a lot of results in the literature; see Remark \ref{AR conjecture over normal ring / improve and recover} for details.

Auslander and Reiten \cite{AR} conjectured that $M$ is projective if $\Ext_R^i(M,R\oplus M)=0$ for all $i>0$.
This is known as the \textit{Auslander--Reiten conjecture}, which is a celebrated long-standing conjecture and one of the most important problems in ring theory.
Applying the above theorem, we can give an answer to this conjecture.

\begin{cor}[Corollary \ref{AR conjecture over normal ring}]\label{AR conjecture (S2) introduction2}
Suppose that $R$ satisfies $(\S_2)$.
Then the Auslander--Reiten conjecture holds for $R$ if it locally holds for $R$ in codimension one.
In particular, the Auslander--Reiten conjecture holds true for every normal ring.
\end{cor}

The Auslander--Reiten conjecture is known to hold if $R$ is a locally excellent Cohen--Macaulay normal ring containing the field of rational numbers $\Q$ \cite{HL}, or if $R$ is a Gorenstein normal ring \cite{Ar}, or if $R$ is a local Cohen--Macaulay normal ring and $M$ is a maximal Cohen--Macaulay module of rank one \cite{GT}, or if $R$ is a local Cohen--Macaulay normal ring and $M$ is a maximal Cohen--Macaulay module such that $\Hom_R(M,M)$ is free \cite{DEL}.
Recently, Kimura, Otake, and Takahashi \cite{KOT} proved the conjecture for every Cohen--Macaulay normal ring, which recovers all the above results.
Even if $R$ is not Cohen--Macaulay, it is known that $R$ satisfies the conjecture if it is a normal ring and either $\Ext_R^i(\Hom_R(M,M),R)=0$ for all $2\le i\le\depth R$ or $\Hom_R(M,M)$ has finite G-dimension \cite{ST}, or if it is a quotient of a regular local ring and is a normal ring containing $\Q$ \cite{DEL}.
Corollary \ref{AR conjecture (S2) introduction2} improves all of these results about the Auslander--Reiten conjecture.

All of our results, including the above ones, are discussed in Section 2.
It is worth mentioning that we shall prove most of the results of Kimura, Otake and Takahashi \cite{KOT} without assuming that the ring is Cohen--Macaulay.
Throughout \cite{KOT}, the existence of a canonical module can be assumed since the ring is Cohen--Macaulay.
Some properties of a canonical module, such as the finiteness of injective dimension and the duality of the category of maximal Cohen--Macaulay modules given by the canonical dual, do play a crucial role in [10].
As we do not assume Cohen--Macaulayness, we cannot use those methods in this paper.
However, we can achieve our purpose by investigating the structure of the tail part of a dualizing complex in detail.

We close the section by stating our convention.

\begin{conv}
Let $n\ge 0$ be an integer.
We say that $R$ satisfies {\em Serre's condition} $(\S_n)$ if $\depth R_\p \ge {\rm inf}\{n, \height\p \}$ for all prime ideals $\p$ of $R$.
Denote by $E_R(M)$ the injective hull of $M$.
We denote by $(-)^\ast$ the $R$-dual $\Hom_R(-,R)$.
For a property $\P$ of local rings (resp. of modules over a local ring) and a subset $X$ of $\spec R$, we say that $R$ (resp. an $R$-module $M$) {\em locally satisfies $\P$ on $X$} if the local ring $R_\p$ (resp. the $R_\p$-module $M_\p$) satisfies $\P$ for every $\p\in X$.
When $X=\{\p\in\spec R\mid\height\p\le n\}$, we say that $R$ (resp. $M$) {\em locally satisfies $\mathbb{P}$ in codimension $n$}. 
Let $R$ be a local ring, and let $F=(\cdots\to F_2\to F_1\xrightarrow{\alpha}F_0\to0)$ be a minimal free resolution of $M$.
The {\em syzygy} $\syz_R M$ and the {\em (Auslander) transpose} $\tr_R M$ of $M$ is defined as $\image \alpha$ and $\cok(\alpha^\ast)$, respectively.
We put $\syz_R^0M=M$, and the {\em $n$th syzygy} $\syz_R^n M$ is defined inductively by $\syz_R^n M=\syz_R(\syz_R^{n-1}M)$.
(Note that the modules $\tr_R M$ and $\syz_R^nM$ for each $n$ are uniquely determined up to isomorphism since so is a minimal free resolution $F$ of $M$ and that $M^\ast\cong\syz_R^2\tr_R M$ up to free summands.)
We omit subscripts/superscripts if there is no ambiguity.
\end{conv}

\section{The results}

In this section, we provide several criteria for a module to be projective in terms of Ext vanishing, and prove Theorem \ref{AR conjecture (S2) introduction1} and Corollary \ref{AR conjecture (S2) introduction2} stated in the Introduction.
We prepare several lemmas to state the main results of this paper.

\begin{lem}\label{Tor Ext injective}
Let $N$ be an $R$-module, and let $I$ be an injective $R$-module.
Then there is an isomorphism $\Tor_i^R(M,\Hom(N,I))\cong\Hom(\Ext_R^{i}(M,N),I)$ for every integer $i\ge 0$.
\end{lem}

\begin{proof}
There is a quasi-isomorphism $M\otimes_R^{\mathbf{L}} \RHom_R(N,I)\simeq\RHom_R(\RHom_R(M,N),I)$.
\end{proof}

\begin{lem}\label{Functor injective support prime zero}
Let $F$ be an $R$-linear functor on the category of $R$-modules, and let $\p$ be a prime ideal of $R$.
If $F(E_R(R/\p))_\p$ is the zero module, then so is $F(E_R(R/\p))$.
\end{lem}

\begin{proof}
Suppose $F(E_R(R/\p))\ne 0$.
There is a finitely generated nonzero submodule $N$ of $F(E_R(R/\p))$.
Since $N_\p=0$, we can choose an element $r\in R\setminus\p$ such that $rN=0$.
The multiplication by $r$ on $E_R(R/\p)$ is an isomorphism, and thus so is the multiplication by $r$ on $F(E_R(R/\p))$, which is a contradiction.
\end{proof}

The following lemma improves \cite[Proposition 3.3(1)]{KOT}.
Assuming the non-vanishing of the tensor product $k\otimes_R N$, the assumption that $N$ is finitely generated is removed.

\begin{lem}\label{trans tor1 free lemma non f.g.}
Let $(R, \m, k)$ be a local ring, and let $N$ be an $R$-module such that $k\otimes_R N$ is nonzero.
Suppose that $\Tor_1^R(\tr M, M\otimes_R N)=0$.
Then $M$ is a free $R$-module.
\end{lem}

\begin{proof}
We may assume that $M$ is nonzero and indecomposable.
Let $\lambda:M\otimes M^\ast\to R$ be the homomorphism given by $\lambda(x\otimes f)=f(x)$ for $x\in M$ and $f\in M^\ast$.
If $\lambda$ is surjective, then $R$ is isomorphic to a direct summand of $M$, and thus we have $M\cong R$ as $M$ is indecomposable; see \cite[Proposition 2.8(iii)]{L} for instance.
Therefore, it suffices to show $\image\lambda=R$.
Set $L = M\otimes_R N$.
An analogous argument to the proof of \cite[Proposition 3.3(1)]{KOT} shows that $(R/\image\lambda)\otimes_R L=0$.
On the other hand, since $k\otimes_R N$ and $M$ are nonzero, so is $k\otimes_R L$.
It means that $\image\lambda=R$.
\end{proof}

One of the main results of this paper is the theorem below.
It plays an essential role in proving Theorem \ref{AR conjecture (S2) introduction1}.

\begin{thm}\label{punc spec free depth t}
Let $(R, \m, k)$ be a local ring of depth $t$.
Suppose that $\Ext_R^i(M,R)=0$ for all $1\le i\le t$ and $\Ext_R^{t+1}(\tr M,M^{\ast})=0$, and that $M$ is locally free on the punctured spectrum of $R$.
Then $M$ is free.
\end{thm}

\begin{proof}
We may assume that $R$ is complete.
There is a Gorenstein local ring $S$ of dimension $d=\dim R$ such that $R$ is a homomorphic image of $S$.
Let $I:0\to I^0\to I^1\to\cdots I^d\to 0$ be a mininal injective resolution of $S$-module $S$.
Set $D=\Hom_S(R,I): 0\to D^0\xrightarrow{\partial^0} D^1\xrightarrow{\partial^1} \cdots D^d\xrightarrow{\partial^d}  0$, and $K^i=\ker \partial^i$ for each $0\le i\le d$. 
(The complex $D$ is called a dualizing complex.)
It is easy to see that 
\begin{equation}\label{punc spec free depth t no.1}
D^i\cong\bigoplus_{\substack{\p\in\spec R \\ \dim R/\p=d-i}} E_R(R/\p)
\quad {\rm and} \quad 
H^j(D)=\Ext_S^j(R,S)=0
\end{equation}
for all $0\le i\le d$ and all $d-t+1\le j\le d$.
Hence, for any $1\le i\le t$, there exists an exact sequence
\begin{equation}\label{punc spec free depth t no.2}
0 \to K^{d-i} \to D^{d-i} \to K^{d-i+1} \to 0.
\end{equation}

Note that by Lemma \ref{Tor Ext injective}, for any injective $R$-module $J$ and any $0\le j\le t$, we have $\Tor_j^R(M.J)\cong\Hom(\Ext_R^{j}(M,R),J)=0$.
We claim by induction on $i$ that $\Tor_j^R(M.K^{d-i})=0$ for all $0\le i\le t-1$ and all $1\le j\le t-i$.
Indeed, if $i=0$, then $K^d=D^d$ is an injective $R$-module.
Let $1\le i\le t-1$.
It follows from (\ref{punc spec free depth t no.2}) that for each $0\le j\le t-i$, there is an exact sequence $\Tor_{j+1}^R(M,K^{d-i+1})\to \Tor_{j}^R(M,K^{d-i}) \to\Tor_{j}^R(M,D^{d-i})$ of $R$-modules.
By the induction hypothesis, the claim holds.

This claim shows that for each $1\le i\le t$, the natural complex
\begin{equation}\label{punc spec free depth t no.3}
0 \to M\otimes_R K^{d-i} \to M\otimes_R D^{d-i} \to M\otimes_R K^{d-i+1} \to 0
\end{equation}
is an exact sequence.
Also, we claim that $\Tor_{t+1-i}^R(\tr M, M\otimes_R K^{d-i})=0$ for all $0\le i\le t$.
We use induction on $i$.
If $i=0$, then by Lemma \ref{Tor Ext injective}, we obtain isomorphisms
$$
\Tor_{t+1}^R(\tr M, M\otimes_R K^{d})\cong
\Tor_{t+1}^R(\tr M, \Hom(M^{\ast},K^{d}))\cong
\Hom(\Ext_R^{t+1}(\tr M,M^{\ast}),K^{d})=0.
$$
Suppose $1\le i\le t$.
The exact sequence (\ref{punc spec free depth t no.3}) induces an exact sequence 
$$
\Tor_{t+2-i}^R(\tr M, M\otimes_R K^{d-i+1})\to \Tor_{t+1-i}^R(\tr M, M\otimes_R K^{d-i}) \to\Tor_{t+1-i}^R(\tr M, M\otimes_R D^{d-i}).
$$
By (\ref{punc spec free depth t no.1}) and Lemma \ref{Functor injective support prime zero}, we see that $\Tor_{t+1-i}^R(\tr M, M\otimes_R D^{d-i})=0$ since $M$ is locally free on the punctured spectrum of $R$.
So, the claim holds by the induction hypothesis.
In particular, we have $\Tor_1^R(\tr M, M\otimes_R K^{d-t})=0$.

The module $\Tor_t^R(k, K^d)\cong\Hom(\Ext_R^{t}(k,R),K^{d})$ is nonzero as $\depth R=t$.
The exact sequences (\ref{punc spec free depth t no.2}) induce isomorphisms $\Tor_t^R(k, K^d)\cong\Tor_{t-1}^R(k, K^{d-1})\cong\cdots\cong\Tor_1^R(k, K^{d-t+1})\cong k\otimes_R K^{d-t}$ by (\ref{punc spec free depth t no.1}) and Lemma \ref{Functor injective support prime zero}.
Therefore $k\otimes_R K^{d-t}$ is a nonzero module.
Lemma \ref{trans tor1 free lemma non f.g.} implies that $M$ is free.
\end{proof}

Note that if $\Ext_R^2(\tr M,M^{\ast})=0$, then the natural homomorphism from $M\otimes_R M^\ast$ to $\Hom_R(M^\ast,M^\ast)$ is surjective; see \cite[Proposition (2,6)(a)]{AB}.
It is easy to see that $M^\ast$ is free, and thus so is $M$.
Therefore, the assertion of Theorem \ref{punc spec free depth t} is easily shown without using our techniques of dualizing complexes when $t=\depth R=1$.
It is the case where $t\ge 2$ which is important for us.
In this case, we have an isomorphism $\Ext_R^{t+1}(\tr M,M^{\ast})\cong\Ext_R^{t-1}(M^{\ast},M^{\ast})$ since $\syz^2\tr M$ is isomorphic to $M^\ast$ up to free summands, and thus the above theorem can be rewritten as Corollary \ref{punc spec free depth 2 ijou R-dual ver}.

\begin{cor}\label{punc spec free depth 2 ijou R-dual ver}
Let $R$ be a local ring of depth $t\ge 2$.
Suppose that $\Ext_R^i(M,R)=0$ for all $1\le i\le t$ and $\Ext_R^{t-1}(M^{\ast},M^{\ast})=0$, and that $M$ is locally free on the punctured spectrum of $R$.
Then $M$ is free.
\end{cor}

To replace the condition on the vanishing of a self-extension of $M^\ast$ with that of $M$, one elementary lemma is needed to be prepared.

\begin{lem}\label{ext syzygy shift lemma}
Let $R$ be a local ring, and let $s, t>0$ be integers.
Suppose that $\Ext_R^i(M,R)=0$ for all $s\le i\le s+t$.
Then $\Ext_R^s(M,M)\cong\Ext_R^s(\Omega^t M,\Omega^t M)$.
\end{lem}

\begin{proof}
We may assume that $t=1$.
The natural exact sequence $0\to \Omega M\to F \to M\to 0$, where $F$ is a free $R$-module induces an exact sequence $\Ext_R^s(M,F)\to\Ext_R^s(M,M)\to\Ext_R^{s+1}(M,\Omega M)\to\Ext_R^{s+1}(M,F)$.
We have $\Ext_R^s(M,M)\cong\Ext_R^{s+1}(M,\Omega M)\cong\Ext_R^s(\Omega M,\Omega M)$ by assumption.
\end{proof}

\begin{rmk}\label{fpd iff free}
Suppose that $R$ is local and $M$ has finite projective dimension.
If $\Ext_R^i(M,R)=0$ for all $1\le i\le \depth R$ or $\depth M\ge\depth R$, then it follows from \cite[Page 154, Lemma 1(iii)]{Mat} and the Auslander--Buchsbaum formula that $M$ is free.
\end{rmk}

Let $m,n \in\mathbb{Z}_{\ge 0}\cup\{\infty\}$.
Denote by $\G_{m,n}$ the full subcategory of the category of finitely generated $R$-modules consisting of modules $M$ such that $\Ext^i_R(M,R)=0$ for all $1\le i\le m$ and $\Ext^j_R(\tr M,R)=0$ for all $1\le j\le n$.

\begin{prop}\label{punc spec free depth 2 ijou}
Let $R$ be a local ring of depth $t\ge 2$.
Suppose that $\Ext_R^{t-1}(M,M)=0$ and $M$ is locally free on the punctured spectrum of $R$.
Then $M$ is free in each of the two cases below.

{\rm (1)} $M \in \G_{2t-r+1,r}$ for some $0\le r\le 2t+1$. \quad
{\rm (2)} $\Ext_R^i(M,R)=0$ for all $1\le i\le 2t+1$.
\end{prop}

\begin{proof}
We have only to show the proposition in case (1).
Set $s={\rm max}\{ t-r+2, 0 \}$.
We obtain inequalities $t+2\le r+s\le 2t+1$.
Since $M \in \G_{2t-r+1,r}$, we have $N:=\Omega^{s}M\in \G_{2t-r-s+1,r+s}\subseteq\G_{0,2}$ and $\syz^2\tr N\in \G_{r+s-2, 2t-r-s+3}\subseteq\G_{t,0}$ by \cite[Proposition 1.1.1]{I}.
Thus $N$ is reflexive by \cite[Proposition (2.6)(a)]{AB}.
As $N^{\ast}$ is isomorphic to $\syz^2\tr N$ up to free summands, we get $\Ext_R^i(N^{\ast},R)=0$ for all $1\le i\le t$.
Lemma \ref{ext syzygy shift lemma} yields that $\Ext_R^{t-1}((N^\ast)^\ast,(N^\ast)^\ast)\cong\Ext_R^{t-1}(N,N)\cong\Ext_R^{t-1}(M,M)=0$.
By Corollary \ref{punc spec free depth 2 ijou R-dual ver}, $N^\ast$ is free, and so is $N\cong(N^\ast)^\ast$.
It follows from Remark \ref{fpd iff free} that $M$ is free when $s\ne 0$.
\end{proof}

\begin{rmk}\label{Tate cohomology reflexive free}
In Proposition \ref{punc spec free depth 2 ijou}, the condition $t=\depth R\ge 2$ is imposed as an assumption.
The difference between the cases $t\ge 2$ and $t\le 1$ can also be considered as follows.
Suppose that $M\in\G_{\infty,\infty}$ and the Tate cohomology module $\widehat\Ext{}^{t-1}_R(M,M)$ is zero.
We see that $\Ext_R^{t-1}(M,M)=0$ when $t\ge 2$; see \cite{R} for instance.
In other words, the condition $\widehat\Ext{}^{t-1}_R(M,M)=0$ can be explained in terms of vanishing of a self-extension of $M$.
When $t\le 1$, it cannot be done as such.
However, a relationship exists between this condition and Theorem \ref{punc spec free depth t}.
If $t=1$, then we have $\underline{\Hom}{}_R(M,M)\cong\widehat\Ext{}^{0}_R(M,M)=0$, and thus $M$ is free.
Let $t=0$.
Then we get isomorphisms $0=\widehat\Ext{}^{-1}_R(M,M)\cong\widehat\Ext{}^{-1}_R(\Omega^2 M,\Omega^2 M)\cong\Ext{}^{1}_R(M,\Omega^2 M)\cong\Ext{}^{1}_R(\tr N,N^\ast)$, where $N=\tr M$.
This Ext module is the same as the one in Theorem \ref{punc spec free depth t}.
In particular, Theorem \ref{punc spec free depth t} and Proposition \ref{punc spec free depth 2 ijou} correspond to \cite[Theorem 8]{Ar} when $R$ is Gorenstein.
\end{rmk}

Let $n\ge 0$ be an integer.
We denote by $\x^n(R)$ the set of prime ideals of $R$ with height at most $n$.
The proposition below gives a generalization of Corollary \ref{punc spec free depth 2 ijou R-dual ver} and Proposition \ref{punc spec free depth 2 ijou}.
Needless to say, (4) is a special case of (3).

\begin{prop}\label{S2 free subset Spec(R)}
Let $X$ be a subset of $\spec R$ containing $\x^1(R)$.
Put $s:={\rm inf}\{ \depth R_\p \mid \p\in \spec R\setminus X \}$ and $t:={\rm sup}\{ \depth R_\p \mid \p\in \spec R\setminus X \}$.
Suppose that $R$ satisfies $(\S_2)$ and that $M$ is locally free on $X$.
Then $M$ is projective if one of the following four conditions is satisfied.
\begin{enumerate}[\rm(1)]
\item
$\Ext_R^i(M,R)=\Ext_R^j(M^{\ast},M^{\ast})=0$ for all $1\le i\le t$ and $s-1\le j\le t-1$.
\item
$\Ext_R^i(M^{\ast},R)=\Ext_R^j(M,M)=0$ for all $1\le i\le t$ and $s-1\le j\le t-1$, and $M$ satisfies $(\S_2)$.
\item
$M \in \G_{m,n}$ for some $m,n\ge 0$ such that $m+n=2t+1$, and $\Ext_R^j(M,M)=0$ for all $s-1\le j\le t-1$.
\item
$\Ext_R^i(M,R)=\Ext_R^j(M,M)=0$ for all $1\le i\le 2t+1$ and $s-1\le j\le t-1$.
\end{enumerate}
\end{prop}

\begin{proof}
It is enough to show that $M_\p$ is a free $R_\p$-module for all prime ideals $\p$ of $R$.
So, we may assume that $R$ is local.
We prove the corollary by induction on $d=\dim R$.
We know that the assertion holds true for $d\le 1$ by assumption.
Suppose $d\ge 2$.
We may assume that $X\ne \spec R$.
By the induction hypothesis, $M$ is locally free on the punctured spectrum of $R$.
Since $R$ satisfies $(\S_2)$, we have $\depth R\ge 2$.
In cases (1), (3) and (4), Corollary \ref{punc spec free depth 2 ijou R-dual ver} and Proposition \ref{punc spec free depth 2 ijou} conclude that $M$ is free, respectively.
In case (2), by \cite[Theorem 3.6]{EG}, we see that $M$ is reflexive.
Hence (1) implies that $M^\ast$ is free, and thus so is $M$.
\end{proof}

\begin{rmk}\label{CM version of S2 free subset Spec(R)}
The ring $R$ is called \textit{Cohen-Macaulay} if $\depth R_\p\ge \dim R_\p$ for every prime ideal $\p$ of $R$. 
Let $R$ be a $d$-dimensional Cohen-Macaulay ring.
For any integer $1\le n\le d-1$, it is easy to see that ${\rm inf}\{ \depth R_\p \mid \p\in \spec R\setminus \x^n(R) \}=n+1$ and ${\rm sup}\{ \depth R_\p \mid \p\in \spec R\setminus \x^n(R) \}=d$.
\end{rmk}

In particular, noting Remark \ref{fpd iff free}, we obtain Theorem \ref{AR conjecture (S2) introduction1} mentioned in Section 1 as a simple corollary of Proposition \ref{S2 free subset Spec(R)}.

\begin{cor}\label{S2 free height less than 2}
Suppose that $R$ satisfies $(\S_2)$ and $M$ is locally of finite projective dimension on $\x^1(R)$.
Then $M$ is projective if one of the following four conditions is satisfied, where $d=\dim R$.
\begin{enumerate}[\rm(1)]
\item
$\Ext_R^i(M,R)=\Ext_R^j(M^{\ast},M^{\ast})=0$ for all $1\le i\le d$ and $1\le j\le d-1$.
\item
$\Ext_R^i(M^{\ast},R)=\Ext_R^j(M,M)=0$ for all $1\le i\le d$ and $1\le j\le d-1$, and $M$ satisfies $(\S_2)$.
\item
$M \in \G_{m,n}$ for some $m,n\ge 0$ such that $m+n=2d+1$, and $\Ext_R^j(M,M)=0$ for all $1\le j\le d-1$.
\item
$\Ext_R^i(M,R)=\Ext_R^j(M,M)=0$ for all $1\le i\le 2d+1$ and $1\le j\le d-1$.
\end{enumerate}
\end{cor}

The following two corollaries can be shown similarly as in the proof of \cite[Corollaries 3.10 and 3.11]{KOT}, respectively.

\begin{cor}\label{recover of KOT no Corollary1}
Suppose that $R$ satisfies $(\S_2)$ and is locally a complete intersection in codimension one.
Then $M$ is projective in each of the following two cases, where $d=\dim R$.
\begin{enumerate}[\rm(1)]
\item
$\Ext_R^i(M^{\ast},R)=\Ext_R^j(M,M)=0$ for all $1\le i\le d$ and $1\le j\le \max\{2,d-1\}$, and $M$ satisfies $(\S_2)$.
\item
$\Ext_R^i(M,R)=\Ext_R^j(M,M)=0$ for all $1\le i\le 2d+1$ and $1\le j\le \max\{2,d-1\}$.
\end{enumerate}
\end{cor}

\begin{proof}
A similar argument to the proof of \cite[Corollary 3.10]{KOT} shows that $M$ is locally free on $\x^1(R)$.
The assertions follow from Corollary \ref{S2 free height less than 2}.
\end{proof}

\begin{cor}\label{recover of KOT no Corollary2}
Let $S$ be a locally excellent ring of dimension $d$ which is locally a complete intersection in codimension one.
Let $\xx=x_1,\dots,x_n$ be a sequence of elements of $S$ which is locally regular on $\V(\xx)$.
Assume that $R=S/(\xx)$ and that $S$ locally satisfies $(\S_2)$ on $\V(\xx)$.
Suppose that $\Ext_R^i(M,R)=\Ext_R^j(M,M)=0$ for all $1\le i\le 2d+1$ and $1\le j\le\max\{2,d-1\}$.
Then $M$ is projective.
\end{cor}

\begin{proof}
Similarly, as in the proof of \cite[Corollary 3.11]{KOT}, it is seen that $M$ is projective; replace \cite[Corollary 3.10(2)]{KOT} with Corollary \ref{recover of KOT no Corollary1}(2) in the proof of \cite[Corollary 3.11]{KOT}.
\end{proof}

Recall a celebrated long-standing conjecture due to Auslander and Reiten \cite{AR}.

\begin{conj}[Auslander--Reiten]\label{AR conjecture}
Every $R$-module $M$ such that $\Ext^{>0}_R(M,M\oplus R)=0$ is projective.
\end{conj}

Corollary \ref{S2 free height less than 2} yields the following result, which gives affirmative answers to Conjecture \ref{AR conjecture}.

\begin{cor}\label{AR conjecture over normal ring}
Suppose that $R$ satisfies $(\S_2)$.
\begin{enumerate}[\rm(1)]
\item
If $\Ext^{>0}_R(M,M\oplus R)=0$ and $M$ is locally of finite projective dimension in codimension one, then $M$ is projective.
\item
Conjecture \ref{AR conjecture} holds for $R$ if it locally holds for $R$ in codimension one.
\item
Conjecture \ref{AR conjecture} holds true for every normal ring.
\end{enumerate}
\end{cor}

\begin{proof}
We may assume that $R$ is local.
The assertion (1) follows from Corollary \ref{S2 free height less than 2}(4) and yields that (2) holds.
Since an arbitrary regular local ring satisfies Conjecture \ref{AR conjecture}, (2) implies that (3) holds.
\end{proof}

The results obtained in this paper refine (or recover) a lot of results in the literature.

\begin{rmk}\label{AR conjecture over normal ring / improve and recover}
\begin{enumerate}[(1)]
\item
Proposition \ref{punc spec free depth 2 ijou}(2) is a non-Gorenstein version of \cite[Corollary 10]{Ar}.
Indeed, let $R$ be a Gorenstein ring of dimension $d\ge 2$.
It is seen that $\Ext_R^i(M,R)=0$ for all $i>d$ and that $M$ is maximal Cohen--Macaulay if and only if for all $1\le j\le d$, $\Ext_R^j(M,R)=0$.
\item
Proposition \ref{S2 free subset Spec(R)}(3) (resp. (4)) improves \cite[Theorem 3.9(1)]{KOT} (resp. \cite[Theorem 3.9(2)]{KOT}), that is to say, Proposition \ref{S2 free subset Spec(R)}(3) (resp. (4)) relaxes the assumption of $R$ being Cohen--Macaulay to $R$ satisfying $(\s_2)$; see Remark \ref{CM version of S2 free subset Spec(R)}.
\item
Proposition \ref{S2 free subset Spec(R)}(4) highly refines \cite[Corollary 1.6]{secm}.
Let $R$ and $M$ be as in \cite[Corollary 1.6]{secm} for some $1\le n\le d-1$.
Similarly as (1), we have $\Ext_R^i(M,R)=0$ for all $i>0$, and thus $M$ is free by Proposition \ref{S2 free subset Spec(R)}(4) and Remark \ref{CM version of S2 free subset Spec(R)}.
\item
Proposition \ref{S2 free subset Spec(R)}(4) partly refines \cite[Corollary 1.5]{ST}.
When $R$ satisfies $(\s_2)$, \cite[Corollary 1.5]{ST} asserts almost the same as Proposition \ref{S2 free subset Spec(R)}(4) for $X=\{ \p\in\spec R \mid$ The $R_\p$-module $M_\p$ is free$ \}$ under the assumption that $R$ is generically Gorenstein, that $\Ext^{>0}_R(M,R)=0$, and that either $\Ext_R^i(\Hom_R(M,M),R)=0$ for all $n<i\le\depth R$ or $\Hom_R(M,M)$ has finite G-dimension.
\item
Corollary \ref{S2 free height less than 2}(4) reaches the same conclusion as \cite[Theorem 3.14]{DEL} by assuming more vanishings of Ext modules and instead removing the assumption that $R$ is a quotient of a regular local ring that is locally Gorenstein in codimension one and contains $\Q$; see Remark \ref{fpd iff free}.
\item
Corollary \ref{recover of KOT no Corollary1} asserts the same as \cite[Theorem 3.10]{KOT} under the assumption that the ring $R$ satisfies $(\s_2)$, which is weaker than $R$ being Cohen--Macaulay.
In particular, Corollary \ref{recover of KOT no Corollary1}(1) considerably refines \cite[Theorem 1.3]{HL}, and assuming more vanishings of Ext modules, (2) highly improves \cite[Theorem 3.16]{DEL}; see \cite[(5) and (6) of Remark 3.14]{KOT}.
\item
The same conclusion as \cite[Theorem 1.5(2)]{GT} is deduced from Corollary \ref{recover of KOT no Corollary1}(2) by assuming more vanishings of Ext modules and instead removing the assumption that $R$ is Cohen--Macaulay, and that $M$ is a maximal Cohen--Macaulay $R$-module of rank one, and weakening normality to the local complete intersection property in codimension one.
\item
Assuming appropriate vanishings of Ext modules,
we can remove by Corollary \ref{recover of KOT no Corollary1}(2) the assumptions imposed in \cite[Corollary 1.6]{ST} that either $\Ext_R^i(\Hom_R(M,M),R)=0$ for all $2\le i\le\depth R$ or $\Hom_R(M,M)$ has finite G-dimension.
\item
Suppose that $S$, $\xx$, and $R$ are as in \cite[Corollary 3.11]{KOT}.
It follows that $S$ locally satisfies $(\S_2)$ on $\V(\xx)$ since $R=S/(\xx)$ is Cohen--Macaulay and $\xx$ is locally regular on $\V(\xx)$.
Hence, Corollary \ref{recover of KOT no Corollary2} improves \cite[Corollary 3.11]{KOT} and thus \cite[Main Theorem]{HL}; see \cite[Remark 3.14(7)]{KOT}.
\item
Simiarly as in the previous item (9), we see that (1) and (2) of Corollary \ref{AR conjecture over normal ring} extend \cite[(1) and (2) of Corollary 3.13]{KOT} from Cohen--Macaulay rings to the rings which satisfy $(\s_2)$.
\item
Conjecture \ref{AR conjecture} is known to hold for every Cohen--Macaulay normal ring by virtue of \cite[Corollary 1.3]{KOT}.
Corollary \ref{AR conjecture over normal ring}(3) shows that the conjecture also holds for an arbitrary normal ring, i.e. it refines \cite[Corollary 1.3]{KOT}.
\item
Conjecture \ref{AR conjecture} holds true if $R$ is a quotient of a regular local ring and is a normal ring containing $\Q$ \cite[Theorem 3.14]{DEL}.
In particular, every complete normal local ring of equicharacteristic zero satisfies the conjecture.
Note that since the normality is not necessarily stable under completion, it is not easy to remove from these results the assumption that $R$ is a quotient of a regular local ring or $R$ is complete.
Corollary \ref{AR conjecture over normal ring}(3), however, does make it happen.
\end{enumerate}
\end{rmk}

\begin{ac}
The author would like to thank his supervisor Ryo Takahashi for valuable comments.
\end{ac}

\end{document}